\documentclass[10pt]{amsart}

\usepackage{cite}
\usepackage[pdftex]{hyperref}
\usepackage{amsmath,amssymb,amsthm,textcomp}

\theoremstyle{plain}
\newtheorem{theorem}{Theorem}[section]
\newtheorem{lemma}[theorem]{Lemma}
\newtheorem{corollary}[theorem]{Corallary}
\newtheorem*{stht}{Skew-torsion holonomy theorem}
\theoremstyle{definition}
\newtheorem{definition}[theorem]{Definition}
\theoremstyle{remark}
\newtheorem{remark}[theorem]{Remark}
\newtheorem*{notation}{Notation}
\newtheorem{example}[theorem]{Example}

\DeclareMathOperator{\Ad}{Ad}
\DeclareMathOperator{\Aff}{Aff}
\DeclareMathOperator{\diag}{diag}
\DeclareMathOperator{\Exp}{Exp}
\DeclareMathOperator{\grad}{grad}
\DeclareMathOperator{\Hol}{Hol}
\DeclareMathOperator{\id}{id}
\DeclareMathOperator{\Iso}{Iso}
\DeclareMathOperator{\Lie}{Lie}
\DeclareMathOperator{\SO}{SO}
\DeclareMathOperator*{\sspan}{span}
\DeclareMathOperator{\Spin}{Spin}
\DeclareMathOperator{\SU}{SU}

\renewcommand{\gg}{\mathfrak g}
\newcommand{\gh}{\mathfrak h}
\newcommand{\hol}{\mathfrak{hol}}
\newcommand{\gk}{\mathfrak k}
\newcommand{\gm}{\mathfrak m}
\newcommand{\so}{\mathfrak{so}}
\newcommand{\tr}{\mathfrak{tr}}
\newcommand{\gX}{\mathfrak X}
\newcommand{\bbr}{\mathbb R}
\newcommand{\bbv}{\mathbb V}

\begin{document}

\title[A Berger-type theorem]{A Berger-type theorem for metric connections with skew-symmetric torsion}
\author{Silvio Reggiani}
\address{Facultad de Matem\'atica, Astronom\'ia y F\'isica, Universidad Nacional de C\'ordoba, Ciudad Universitaria, 5000 C\'ordoba, Argentina}
\email{reggiani@famaf.unc.edu.ar}

\date{\today}

\begin{abstract}
We prove a Berger-type theorem which asserts that if the orthogonal subgroup generated by the torsion tensor (pulled back to a point by parallel transport) of a metric connection with skew-symmetric torsion is not transitive on the sphere, then the space must be locally isometric to a Lie group with a bi-invariant metric or its symmetric dual (we assume the space to be locally irreducible). We also prove that a (simple) Lie group with a bi-invariant metric admits only two flat metric connections with skew-symmetric torsion: the two flat canonical connections. In particular, we get a refinement of a well-known theorem by Cartan and Schouten. Finally, we show that the holonomy group of a metric connection with skew-symmetric torsion on these spaces generically coincides with the Riemannian holonomy.
\end{abstract}

\maketitle

\section{Introduction}

The family of metric connections on a Riemannian manifold $M$ which have the same geodesics as the Levi-Civita connection is a distinguished class among the family of all connections on $M$. This family attracted the attention of \'E.\ Cartan in the early 20th century. Since then, many mathematicians have been concerned with its study. In the last few years, these connections have also been studied, in a modern approach, because of its applications to physics (string and superstring theory; see \cite{agricola-2006} and the references therein for some examples).

Let $\nabla$ be the Levi-Civita connection of $M$ and let us consider a metric connection $\tilde\nabla$ on $M$ which has the same geodesics as $\nabla$. It is a well-known fact that this is equivalent to the difference tensor $D = \nabla - \tilde\nabla$ being totally skew-symmetric. That is, it defines a $3$-form by contracting with the metric tensor of $M$. In such a case, the torsion tensor of $\tilde\nabla$ is obtained as a constant multiple of $D$, and so we say that $\tilde\nabla$ is a \textit{connection with skew-symmetric torsion}. One of the more remarkable examples of this kind of geometry is the case of the naturally reductive spaces, endowed with the so-called canonical connection. In a recent work \cite{olmos-reggiani-2012} it is shown that, for irreducible compact naturally reductive spaces, the canonical connection is essentially unique (i.e., provided the space is not a sphere, a real projective space or a Lie group with a bi-invariant metric). Moreover, in a forthcoming joint work with C.\ Olmos, it is proved that the same is true in the non-compact case (provided the space is not the dual of a compact Lie group; this includes the hyperbolic $3$-space, which is the dual of $S^3 = \Spin(3)$). On the other hand, the only geometries admitting a flat metric connection with skew-symmetric torsion are the compact Lie groups and the sphere $S^7$. This fact is due to Cartan and Schouten \cite{cartan-schouten-1926} (see \cite{agricola-friedrich-2010} for a modern proof that does not depend on the classification of the symmetric spaces).

The main goal of this short article is to establish a Berger-type result, Theorem \ref{main}, for connections of this kind. In fact, given a metric connection $\tilde\nabla$ on $M$ with skew-symmetric torsion, we have associated an orthogonal Lie subgroup $H(M, \tilde\nabla) \subset \SO(\dim M)$. Our Berger-type theorem asserts that if $\{e\} \neq H(M, \tilde\nabla) \neq \SO(\dim M)$, then $M$ is isometric to a (simple) Lie group with a bi-invariant metric or its symmetric dual ($M$ is assumed to be complete, simply connected and irreducible). Moreover, if the torsion tensor of $\tilde\nabla$ is invariant under the resulting Lie group, then $\tilde\nabla$ is a canonical connection on $M$. Recall that the group $H(M, \tilde\nabla)$ is obtained by pulling back the $\tilde\nabla$-torsion by (Riemannian) parallel transport to a given point in $M$. Notice the analogous construction for the holonomy group, which is obtained in this way from the curvature tensor (the Ambrose-Singer theorem). We need to deal with the group $H(M, \tilde\nabla)$ instead of $\Hol(\tilde\nabla)$ since this last group does not seem to carry enough information (see the flat examples in Section \ref{4}). In fact, we have to  enlarge $\Hol(\tilde\nabla)$ to $H(M, \tilde\nabla)$ to get a Berger-type theorem. 

The definition of the group $H(M, \tilde\nabla)$ is in the spirit of studying geometries which admit a metric connection with skew-symmetric torsion, and Theorem~\ref{main} characterizes these spaces when the torsion is not generic.

We wish to clarify briefly what we mean by a Berger-type theorem. Informally, it means a result which asserts that if a certain orthogonal subgroup is ``generic'', then, our object is ``symmetric''. The fanciest example is the Berger holonomy theorem \cite{berger-1955,olmos-2005}, which asserts that if the holonomy group of an irreducible Riemannian space is not transitive (on the sphere of the tangent space), then the space must be locally symmetric. Another geometric Berger-type theorem is due to Thorbergsson \cite{thorbergsson-1991,olmos-1993}: if $M$ is a submanifold of the sphere with constant principal curvatures and the normal holonomy group of $M$ acts irreducibly and non-transitively, then $M$ is the orbit of an $s$-representation. In \cite{console-discala-olmos-2011} a Berger-type theorem for complex submanifolds is proved: if $M$ is a complete and full complex submanifold of $\mathbb CP^n$ and the normal holonomy of $M$ is non-transitive, then $M$ is the (projectivized) orbit of an irreducible Hermitian $s$-representation. A famous algebraic Berger-type theorem is the so-called Simons holonomy theorem \cite{simons-1962,olmos-2005b}: every irreducible and non-transitive holonomy system must be symmetric. Recall that Simons theorem implies the Berger theorem.

In order to prove Theorem \ref{main}, we work with the concept of skew-torsion holonomy systems, and we make use of the skew-torsion holonomy theorem \cite{olmos-reggiani-2012,nagy-2007} (which is a Berger-type theorem!). In fact, $[T_pM, D^p, H(M, \tilde\nabla)]$ turns out to be a skew-torsion holonomy system. Notice that the only transitive case for an skew-torsion holonomy system is the full orthogonal group (and this explains the assumption that $H(M, \tilde\nabla) \neq \SO(\dim M)$ in Theorem \ref{main}).

As an application of our results we study the holonomy group of metric connections with skew-symmetric torsion on Lie groups. Let $G$ be a simple Lie group with a bi-invariant metric and let $\tilde\nabla$ be a metric connection with skew-symmetric torsion on $G$. We can summarize the results obtained as follows:

\begin{itemize}
\item In the flat case, $\tilde R = 0$, we have that $\tilde\nabla$ is one of the two flat canonical connections, whose torsion tensor is given by $\tilde T(X, Y) = \pm[X,Y]$ (for these connections left or right invariant vector fields are parallel, depending on the chosen sign). See Theorem \ref{flat}, which is a refinement of the Cartan-Schouten theorem \cite{cartan-schouten-1926, agricola-friedrich-2010}.
\item In the general case, when $\tilde\nabla$ is not flat and $H(G,\tilde\nabla) \neq \SO(\dim G)$ we have that $\Hol(\tilde\nabla) = G$. See Theorem \ref{hol=G}.
\end{itemize}


\section{Preliminaries}

In this section we wish to recall some results on skew-torsion holonomy systems and naturally reductive spaces that we use throughout this article. The general reference for this section is \cite{olmos-reggiani-2012}. 

A \textit{skew-torsion holonomy system} is a triple $[\bbv, \Theta, G]$ where $\bbv$ is a Euclidean space, $G$ is a connected Lie subgroup of $\SO(\bbv)$ and $\Theta$ is a totally skew-symmetric $1$-form which takes values in $\gg = \Lie(G)$. That is to say, $\Theta: \bbv \to \gg$ is linear and such that $(x, y, z) \mapsto \langle \Theta_xy, z\rangle$ is an algebraic $3$-form on $\bbv$. We say that such a triple is \textit{irreducible} if $G$ acts irreducibly on $\bbv$, \textit{transitive} if $G$ acts transitively on the unit sphere of $\bbv$, and \textit{symmetric} if $g_*(\Theta) = \Theta$ for all $g \in G$ (where $g_*(\Theta)_x = g \circ \Theta_{g^{-1}(x)} \circ g^{-1}$).

The definition of skew-torsion holonomy systems is motivated by the holonomy systems introduced by J.\ Simons in \cite{simons-1962}, where he considered an algebraic curvature tensor (instead of a $1$-form) taking values in $\gg$. Skew-torsion holonomy systems arise in a natural way in a geometric context, on considering the difference between two metric connections with the same geodesics as the Levi-Civita connection. There is an analogue to the Simons holonomy theorem, which asserts that an irreducible and non-transitive holonomy system must be symmetric. This result is actually stronger since for skew-torsion holonomy systems the only transitive case is the full orthogonal group $G = \SO(\bbv)$.

\begin{stht}[see \cite{nagy-2007,olmos-reggiani-2012}]
Let $[\bbv, \Theta, G]$, $\Theta \neq 0$, be an irreducible skew-torsion holonomy system with $G \neq \SO(\bbv)$. Then $[\bbv, \Theta, G]$ is symmetric and non-transitive. Moreover, 
\begin{enumerate}
\item $(\bbv, [\cdot,\cdot])$ is an orthogonal simple Lie algebra, of rank at least $2$, with respect to the bracket $[x,y] = \Theta_xy$;
\item $G = \Ad(H)$, where $H$ is the connected Lie group associated with the Lie algebra $(\bbv, [\cdot,\cdot])$;
\item $\Theta$ is unique, up to a scalar multiple.
\end{enumerate}
\end{stht}

Now, we refer briefly to some recent results on naturally reductive spaces. A homogeneous space $M = G/H$ endowed with a $G$-invariant metric is said to be a \textit{naturally reductive space} if there exists a reductive decomposition $\gg = \gh \oplus \gm$, where $\gg = \Lie(G)$, $\gh = \Lie(H)$ and $\gm$ is an $\Ad(H)$-invariant subspace of $\gg$ such that the geodesics through $p = eH$ are given by $\Exp(tX)\cdot p$, $X \in \gm$. The reductive complement $\gm$ induces a so-called canonical connection $\nabla^c$ on $M$. The above definition is equivalent to the fact that the Levi-Civita connection and the canonical connection have the same geodesics.

In \cite{olmos-reggiani-2012} it is proved that the canonical connection is unique (in the compact irreducible case) unless $M$ is isometric to a sphere, a real projective space or a compact Lie group endowed with a bi-invariant metric. As a consequence of this result it follows that the connected component of the $\nabla^c$-affine group (i.e., the subgroup of diffeomorphisms of $M$ that preserve $\nabla^c$) coincides with the connected component of the isometry group, $\Aff_0(\nabla^c) = \Iso_0(M)$, except if $M$ is a sphere or a real projective space. Moreover, if there is an isometry which does not preserve the canonical connection, then $M$ is isometric to a Lie group with a bi-invariant metric.

\begin{remark}[see Theorem 1.1 and Remark 6.1 in \cite{olmos-reggiani-2012}]\label{line}
If $M$ is a simple Lie group endowed with a bi-invariant metric, then the family of canonical connections on $M$ is the affine line
$$\mathcal L = \{t\nabla + (1-t)\nabla^c: t \in \bbr\},$$
where $\nabla$ is the Levi-Civita connection of $M$ and $\nabla^c \neq \nabla$ is a fixed canonical connection on $M$. This is due to the fact that the difference tensor between two canonical connections is unique up to a scalar multiple.
\end{remark}

\section{The Berger-type theorem}

Let $(M, \langle\cdot,\cdot\rangle)$ be a Riemannian manifold. We denote by $\nabla$ the Levi-Civita connection of $M$ and by $\tau_c$ the Riemannian parallel transport along a curve $c$ on $M$. Let $\tilde\nabla$ be a metric connection on $M$. We say that $\tilde\nabla$ is a metric connection with \emph{skew-symmetric torsion} if the $\tilde\nabla$-geodesics coincide with the Riemannian geodesics. It is well-known that this is equivalent to the difference tensor
$$D = \nabla - \tilde\nabla$$
being totally skew-symmetric, that is $(u,v,w) \mapsto \langle D_uv, w\rangle$
defines a $3$-form on $M$. Recall that the torsion tensor $\tilde T$ of $\tilde\nabla$ is obtained as
$$D_XY = -\frac{1}{2}\tilde T(X, Y),$$
for all $X, Y \in \gX(M)$.



\begin{definition} Let $M$ be a Riemannian manifold, and let $\tilde\nabla$ be a metric connection on $M$ with skew-symmetric torsion. Let $p \in M$.
\begin{enumerate}
\item We define $\gh_p(M, \tilde\nabla) \subset \so(T_pM)$ as the Lie subalgebra generated by elements of the form
$$(\tau_c)_*(D_v) := (\tau_c)^{-1} \circ D_v$$
where $c$ is taken among all the piecewise smooth curves $c: [0, 1] \to M$ with $c(0) = p$ and $v$ is arbitrary in $T_{c(1)}M$.
\item We define $H_p(M, \tilde\nabla) \subset \SO(T_pM)$ as the connected Lie subgroup with Lie algebra $\gh_p(M, \tilde\nabla)$. 
\end{enumerate}
\end{definition}

\begin{remark} We remark some points on the above definition.
\begin{enumerate}
\item If $M$ is connected, then the group $H_p(M, \tilde\nabla)$ does not depend on the base point $p$. More precisely, $H_q(M, \tilde\nabla)$ is conjugated to $H_p(M, \tilde\nabla)$ by (Riemannian) parallel transport along any curve joining $p$ with $q$. Sometimes, when it is clear from the context, we will denote the group $H_p(M, \tilde\nabla)$ just by $H(M, \tilde\nabla)$.
\item If $M$ is locally irreducible, then $H_p(M, \tilde\nabla)$ is a closed subgroup of the orthogonal group $\SO(T_pM)$. See Remark \ref{rem:closed} after Theorem \ref{main}.
\item Our definition of $H_p(M, \tilde\nabla)$ is motivated by the Ambrose-Singer theorem \cite{ambrose-singer-1953}. In fact, recall that if $\nabla'$ is a (linear) connection on $M$, then the holonomy algebra $\hol_p(\nabla')$ at $p$ is spanned by elements of the form
$$(\tau_c)_*(R'_{v, w}) = (\tau_c)^{-1} \circ R'_{v, w}$$
where $R'$ is the curvature tensor of $\nabla'$ and $v, w \in T_{c(1)}M$. In our particular case, we find it more interesting to work with the difference tensor $D$, instead of the curvature $\tilde R$, because, as we shall see later, the group $H_p(M, \tilde\nabla)$ carries nontrivial information about the geometry of $(M, \tilde\nabla)$ even in the flat case $\tilde R = 0$. (Recall that in the flat case $M$ is the Riemannian product of Lie groups with bi-invariant metrics and $7$-dimensional spheres of constant curvature \cite{cartan-schouten-1926, agricola-friedrich-2010}.) 
\end{enumerate}
\end{remark}

\begin{remark}[Related concepts in the literature]\label{rem:lit}
We wish to recall that the definition of $H_p(M, \tilde\nabla)$ is actually implicit in the work of Agricola and Friedrich \cite{agricola-friedrich-2004} (see also \cite{nagy-2007, olmos-reggiani-2012}).  In fact, let $\bbv$ be a Euclidean space and let $T \in \Lambda^3(\bbv)$ be an algebraic $3$-form on $\bbv$. In \cite{agricola-friedrich-2004}, 
$$\gg^*_T := \sspan\{T_v: v \in \bbv\} \subset \so(\bbv) \qquad \text{(algebraic span)}$$
and
$$\gh^*_T := [\gg^*_T, \gg^*_T]$$
are defined. We have that $\gg^*_T$ is semisimple, and $\gh^*_T$ coincides with the holonomy algebra of the metric connection with (constant) skew-symmetric torsion on $\bbv$ defined as $\nabla^T = \nabla - T$. These holonomy algebras have been studied exhaustively in \cite{agricola-friedrich-2004} (and also in \cite{nagy-2007, olmos-reggiani-2012}).

Now, let  $\tilde\nabla$ be a metric connection with skew-symmetric torsion on a Riemannian manifold $M$. Observe that, if $T = D^p \in \Lambda^3(T_pM)$ is the difference tensor $D = \nabla - \tilde\nabla$ specialized at the point $p \in M$, then
$$\gh^*_T \subset \gh_p(M, \tilde\nabla).$$
Informally, we say that $\gh_p(M, \tilde\nabla)$ carries more information (for example, $D^p$ could be the zero $3$-form on $T_pM$ for some points).

In the particular case of a naturally reductive space $M = G/H$ and a canonical connection $\nabla^c$ different from the Levi-Civita connection, we have the equality
$$\gh^*_T = \gh_p(M, \nabla^c)$$
for all $p \in M$, since $\nabla^c$ is $G$-invariant (see \cite{olmos-reggiani-2012}, for instance). More generally, if $\tilde\nabla = \nabla - f(\nabla - \nabla^c)$, where $f: M \to \bbr$ is a smooth function, we have that 
$$\gh^*_T = \gh_p(M, \tilde\nabla)$$
provided $f(p) \neq 0$.
\end{remark}


\begin{example}
For metric connections with constant skew-symmetric torsion on $\bbr^n$, i.e., that of the form $\nabla^T = \nabla - T$ where $T \in \Lambda^3(\bbr^n)$, we have that $\gh_p(\bbr^n, \nabla^T) = \gh^*_T$ has been calculated in \cite{agricola-friedrich-2004} for a large family of examples. More generally, $\gh_p(\bbr^n, \tilde\nabla) = \gh^*_T$, for  $\tilde\nabla = \nabla - fT$ where $f: \bbr^n \to \bbr$ is a smooth non-zero function.

In particular, for $n = 4$, $\gh^*_T = \{0\}$ or $\gh^*_T = \so(3)$, since every non-zero $3$-form on $\bbr^4$ is equivalent to $e_{123}$ (and in dimension $3$ there is only one $3$-form up to multiples).
\end{example}

\begin{example}
Let us consider the skew-symmetric $3$-form on $\bbr^4$
$$D = f e_{123} + g e_{234}$$
where $f, g$ are smooth, non-zero, real-valued functions such that $f(x) \neq 0$ implies $g(x) = 0$. If $\tilde\nabla = \nabla - D$, then $\gh_p(\bbr^4, \tilde\nabla) = \so(4) = \so(3) \oplus \so(3)$ for all $p \in \bbr^4$. However, if $T = D^p$, then $\gh^*_T = \{0\}$ or $\gh^*_T = \so(3)$.
\end{example}

\begin{example}[Examples from naturally reductive spaces]
Let $M^n = G/H$ be a simply connected, irreducible naturally reductive space, and let $\nabla^c$ be the associated canonical connection. It follows from \cite{olmos-reggiani-2012} that $\nabla^c$ is unique, except in certain cases (and in such cases the space must be symmetric). More precisely, the canonical connection is unique except when $M = S^n$, or $M$ is a compact Lie group with a bi-invariant metric or its symmetric dual. (Actually, the uniqueness result in \cite{olmos-reggiani-2012} is only for the compact case, and the general case follows from a forthcoming joint work with C.\ Olmos.)

If $M = K$ is a Lie group with a bi-invariant metric, then $H_p(M, \nabla^c) = K$, identified in the natural way. This is explained in detail in Section \ref{4}, but we wish to point out that one can get non-compact examples when $M$ is the symmetric dual of a compact Lie group $K$. In this case, we also have that $H_p(M, \nabla^c) = K$, since canonical connections on $M$ are in a one-one correspondence with canonical connections on $K$.

When $M^n = G/H$ is not a symmetric space or $M = S^n$, we have that $H_p(M, \nabla^c) = \SO(n)$, as it will follow from our main result.
\end{example}




Let us study the group $H_p(M, \tilde\nabla)$ from a holonomic point of view, and the implications of its properties on the geometry of $M$.

First, observe that if $g \in \Hol(\nabla)$, the holonomy group of $M$, then $gH_p(M, \tilde\nabla)g^{-1} \subset H_p(M, \tilde\nabla)$. So,
$$\Hol(\nabla) \subset N(H_p(M, \tilde\nabla)),$$
where $N(H_p(M, \tilde\nabla))$ is the normalizer of $H_p(M, \tilde\nabla)$ in the orthogonal group.

\begin{lemma}\label{H=N(H)}
If $H_p(M, \tilde\nabla)$ acts irreducibly on $T_pM$, then $$H_p(M, \tilde\nabla) = N(H_p(M, \tilde\nabla)).$$
As a consequence, $\Hol(\nabla) \subset H_p(M, \tilde\nabla)$.
\end{lemma}

\begin{proof}
Let $p \in M$ such that $D^p \neq 0$. We just have to observe that if $\Theta = D^p$, then $[T_pM, \Theta, H_p(M, \tilde\nabla)]$ is an irreducible skew-torsion holonomy system. So, by Lemma 3.4 in \cite{olmos-reggiani-2012}  we have that $H_p(M, \tilde\nabla)$ acts on  $T_pM$ as an $s$-representation. Therefore, $H_p(M, \tilde\nabla) = N(H_p(M, \tilde\nabla))$, a well-known fact on $s$-representations (see for example \cite[pp.~192]{berndt-console-olmos-2003}). 
\end{proof}

\begin{remark}
One can prove Lemma \ref{H=N(H)} directly from the skew-torsion holonomy theorem. In fact, for both skew-torsion holonomy systems $[T_pM, \Theta, N(H_p(M, \tilde\nabla))]$ and $[T_pM, \Theta, H_p(M, \tilde\nabla)]$ we have that $H_p(M, \tilde\nabla) = \Ad(G) = N(H_p(M, \tilde\nabla))$, where $G$ is the (simple) Lie group with Lie algebra $(T_pM, [\cdot,\cdot])$, with the Lie bracket given by $[v, w] = \Theta_vw$. But recall that in the proof of the skew-torsion holonomy theorem, the fact that $H_p(M, \tilde\nabla)$ acts as an $s$-representation is used. 
\end{remark}

\begin{remark}
Recall that if $M$ is locally irreducible, then $H_p(M, \tilde\nabla)$ acts irreducibly on $T_pM$. In fact, this is done in the proof of Theorem \ref{main} below. However, $H_p(M, \tilde\nabla)$ could act irreducibly on $T_pM$ even if $M$ splits off (as a Riemannian manifold).

In order to give a counterexample, let us consider on the sphere $S^n$ a canonical connection $\nabla^c \neq \nabla$ and let $D = \nabla - \nabla^c$. Indeed, if $n = 6$ or $n = 7$ we have such canonical connections associated with the nonstandard naturally reductive decompo\-sitions $S^6 = G_2/\SU(3)$ or $S^7 = \Spin(7)/G_2$. Let us consider on $M = S^n \times S^n$ the totally skew-symmetric tensor 
$$\tilde D_{(v,w)}(v',w') = (D_vv' + D_ww', D_w(v'+w'))$$
and the corresponding connection on $M$. It is not hard to see that $H_p(M, \tilde\nabla) = \SO(2n)$. Thus, the irreducibility of the $H_p(M, \tilde\nabla)$-action does not imply that $M$ is irreducible. In fact, we can write
$$\tilde D_{(v,w)} =
\left(\begin{matrix}
D_v & D_w\\
D_w & D_w
\end{matrix}\right).
$$

So, since $M$ is a product, the parallel transport $\tau_c$ along a curve $c$ splits along the projected curves $c_1$ and $c_2$. Thus,
\begin{align*}
(\tau_c^{-1})_*(\tilde D_{(v,w)}) &=
\left(\begin{matrix}
\tau_{c_1}^{-1} & 0\\
0 & \tau_{c_2}^{-1}
\end{matrix}\right)
\left(\begin{matrix}
D_v & D_w\\
D_w & D_w\end{matrix}\right)
\left(\begin{matrix}
\tau_{c_1} & 0\\
0 & \tau_{c_2}
\end{matrix}\right)\\
&= \left(\begin{matrix}
\tau_{c_1}^{-1}D_v\tau_{c_1} & \tau_{c_1}^{-1}D_w\tau_{c_2}\\
\tau_{c_2}^{-1}D_w\tau_{c_1} & \tau_{c_2}^{-1}D_w\tau_{c_2}
\end{matrix}\right)
\end{align*}
and this implies that $\gh_p(M, \tilde\nabla) = \so(2n)$.
\end{remark}

\begin{theorem}\label{main}
Let $M$ be a Riemannian manifold and let $\tilde\nabla$ be a metric connection on $M$ with skew-symmetric torsion $\tilde T$. Assume that $M$ is simply connected, complete and irreducible. If $\{e\} \neq H_p(M,\tilde\nabla) \neq \SO(T_pM)$, then $M$ is isometric to a Lie group with a bi-invariant metric or its symmetric dual. Moreover, if $\tilde T$ is invariant, then $\tilde\nabla$ is a canonical connection on $M$.
\end{theorem}

\begin{proof}
Let $p \in M$ such that $D^p \neq 0$ and let $\Theta = D^p$ be the difference tensor evaluated at $p$. We denote $H = H_p(M, \tilde\nabla)$ to simplify notation. Then $[T_pM, \Theta, H]$ is an irreducible and non-transitive skew-torsion holonomy system. In fact, since $M$ is irreducible, we have that $\Hol(\nabla)$ acts irreducibly on $T_pM$, and then $N(H)$ acts irreducibly on $T_pM$. By making use of the skew-torsion holonomy theorem we get that $N(H)$ is a simple Lie group. Since $H$ is a normal subgroup of $N(H)$ it follows that $H = N(H)$. Thus, the holonomy group is non-transitive on the sphere and $M$ is a symmetric space.

Let $\gg$ be the Lie algebra $(T_pM, [\cdot,\cdot])$, where $[v,w] = \Theta_vw$, and let $G$ be the connected Lie group with Lie algebra $\gg$. Recall that $G$ is isomorphic to $H$ via the adjoint representation $\Ad: G \to H$ (cf.\ Preliminaries and \cite{olmos-reggiani-2012}). Let us consider the bi-invariant metric on $H$ induced by the inner product on $T_pM$.

Let $R$ be the curvature tensor of $M$ evaluated at $p$ and let $\tilde R$ be the curvature tensor of $H$ evaluated at $e \simeq p$. Observe that both $R$ and $\tilde R$ take values in the Lie algebra of $H$. So, $[T_pM, R, H]$ and $[T_pM, \tilde R, H]$ are irreducible holonomy systems in Simons' sense. Therefore, it follows from Theorem 5 in \cite{simons-1962} (see also \cite{olmos-2005b}) that $R = \lambda\tilde R$, for some $\lambda \neq 0$. Just by taking a scalar multiple of the Lie bracket on the Lie algebra $\gh$ of $H$, we may assume that $\lambda = \pm 1$.

If $\lambda = 1$, then by the Cartan-Ambrose-Hicks theorem we have that the identity $\id: T_pM \to T_pM$ extends to an isometry from $M$ onto the universal cover of $H$. On the other hand, if $\lambda = -1$, taking the symmetric dual $M^*$ of $M$, we have that $R^* = -R$. So, with the same argument as before, $M^*$ is isometric to a Lie group endowed with a bi-invariant metric.

Finally, let us fix a canonical connection $\nabla^c \neq \nabla$ on $M$. Then, from the skew-torsion holonomy theorem we have that $D = f(\nabla - \nabla^c)$, where $f: M \to \bbr$ is a differentiable function. Notice that the invariance of $\tilde T$ implies that $D$ is invariant. So, since $\nabla^c$ is an invariant connection, we have for all $v \in T_qM$
$$0 = \nabla^c_v D = v(f)(\nabla - \nabla^c) +f(q)\nabla^c_v(\nabla - \nabla^c) = v(f)(\nabla - \nabla^c).$$
Then $df = 0$, and  thus $f$ is a constant function. This implies that $\tilde\nabla$ is a canonical connection on $M$ (see Remark \ref{line}).
\end{proof}

\begin{remark}\label{rem:closed}
It follows from the proof of the above theorem that $H_p(M, \tilde\nabla)$ is always closed in $\SO(T_pM)$. In fact, if $H_p(M, \tilde\nabla)$ is transitive on $T_pM$, then $H_p(M, \tilde\nabla) = \SO(T_pM)$, by the skew-torsion holonomy theorem. Otherwise, if $\tilde\nabla \neq \nabla$, $H_p(M, \tilde\nabla)$ coincides with the restricted holonomy group of a symmetric space of the group type.
\end{remark}

Taking into account the proof of the above theorem, it makes sense to study the family of connections with skew-symmetric torsion on a compact Lie group $G$, which have the form $\nabla - fD$, where $f:G \to \bbr$ is a differentiable function (actually, we will see later that almost all have this form). We shall do this study in the next section.

\section{The holonomy group of metric connections with skew-symmetric torsion on compact Lie groups}\label{4}

Let $M$ be a (simple) compact Lie group endowed with a bi-invariant metric. We present $M$ as a symmetric space $M = (G \times G) / \diag(G \times G)$. We will denote with no distinction, the Riemannian manifold $M$ by $(G \times G)/\diag(G \times G)$ or just $G$. We also identify, in the natural way, the holonomy group of $M$ with $G$. Recall that the family of canonical connections on $M$ is the $1$-parameter family associated with the naturally reductive complements
$$\gm_\lambda = \{((\lambda + 1)X, (\lambda - 1)X): X \in \gg\}, \qquad \lambda \in \bbr.$$
In particular, $\gm_{0} = \mathfrak p$ in the symmetric decomposition $\gg \oplus \gg = \gk \oplus \mathfrak p$, and in this case the corresponding canonical connection is the Levi-Civita connection of $M$.

\begin{notation}
We will denote by $\nabla^\lambda$ the canonical connection associated with the reductive decomposition
$$\gg \oplus \gg = \diag(\gg \oplus \gg) \oplus \gm_\lambda.$$
The $\nabla^\lambda$-parallel transport along a curve $c$ will be denoted by $\tau_c^\lambda$. In particular, $\nabla^{0} = \nabla$ is the Levi-Civita connection of $M$.
\end{notation}

Notice that both $\nabla^1$ and $\nabla^{-1}$ have trivial holonomy, $\hol(\nabla^1) = \hol(\nabla^{-1}) = \{0\}.$ Moreover, these are the canonical connections that we get from the presentation $M = G/\{e\}$, where the action of $G$ on $M$ is given by left or right multiplication, respectively. For all other canonical connections $\nabla^\lambda$, $\lambda \neq \pm 1$, we have $\Hol(\nabla^\lambda) = \diag(G \times G) \simeq G$. In fact, the holonomy group of any canonical connection $\nabla^\lambda$, $\lambda \in \bbr$, coincides with the isotropy subgroup of the group of transvections of $\nabla^\lambda$. But the Lie algebra of the group of transvections of $\nabla^\lambda$ is given by $\tr(\nabla^\lambda) = [\gm_\lambda, \gm_\lambda] + \gm_\lambda$ (not a direct sum, in general).

From the results in the previous section we have that the group $H(M, \nabla^\lambda)$ associated with the difference tensor $D = \nabla - \nabla^\lambda$, $\lambda \neq 0$, coincides with $G$. In fact, this follows from Theorem \ref{main}, where we proved that $G \simeq H(M, \nabla^\lambda)$ (up to universal cover).

Let us fix $\lambda \neq 0$, and consider the difference tensor $D = \nabla - \nabla^\lambda$. The aim of this section is to study the holonomy group of the family of metric connections with skew-symmetric torsion given by
$$\tilde\nabla^f = \nabla -fD, \qquad f \in C^\infty(G),$$
that is, with difference tensor equal to $fD$.

First of all, recall that if $f$ is a constant function then $\tilde\nabla^f$ is a canonical connection (this is due to the fact that in a simple compact Lie group there exists only a line of canonical connections; see Remark \ref{line}). In particular, for $f \equiv 0$ we have $\tilde\nabla^0 = \nabla^0 = \nabla$, and for $f \equiv 1$ we have $\tilde\nabla^1 = \nabla^\lambda$. Recall the special case in which $\nabla^\lambda = \nabla^{\pm1}$ is a flat canonical connection. In this case, the geodesic symmetry moves $\nabla^\lambda$ into the opposite flat canonical connection $\nabla^{-\lambda} = \nabla^{\mp1}$ (because the geodesic symmetry reverses the sign of the difference tensor; see \cite[Theorem 1.1]{olmos-reggiani-2012}).

\medbreak 

Let $c(t)$ be a curve on $M$ with $c(0) = p$, and denote by $\tau_t$ (resp.\ $\tau_t^\lambda$) the $\nabla$-parallel (resp.\ $\nabla^\lambda$-parallel) transport along $c|_{[0, t]}$.

\begin{lemma}
We have that $\tau_{-t}^\lambda\tau_t$ is a $1$-parameter subgroup of $H_p(M, \nabla^\lambda) \simeq G$.
\end{lemma}

\begin{proof}
In fact, let us show that the curve $\alpha(t) = \tau_{-t}^\lambda\tau_t \in \SO(T_pM)$ is always tangent to $H_p(M, \nabla^\lambda)$. Let $v \in T_pM$ and let $v(t) = \tau_t(v)$ be the parallel transport of  $v$ along $c(t)$. Clearly we have that 
$$\nabla_{c'(t)}^\lambda v(t) = -D_{c'(t)}v(t) = -D_{c'(t)}\tau_t(v).$$
On the other hand,
$$\nabla_{c'(t)}^\lambda v(t) = \frac{\partial}{\partial s}\bigg|_0 \tau_t^\lambda\tau_{-(t+s)}^\lambda\tau_{t+s}(v) = \tau_t^\lambda \frac{d}{dt}\tau_{-t}^\lambda\tau_t(v).$$
So,
$$\frac{d}{dt}\tau_{-t}^\lambda\tau_t = -\tau_{-t}^\lambda D_{c'(t)}\tau_t = -\tau_{-t}^\lambda\tau_t(\tau_{-t}D_{\gamma'(t)}\tau_t) = -\tau_{-t}^\lambda\tau_tD_{c'(0)},$$
since $D$ is a $\nabla$-parallel tensor. This differential equation has a unique solution
$$\alpha(t) = \tau_{-t}^\lambda\tau_t = e^{-tD_{c'(0)}}$$
which is always tangent to $H_p(M, \nabla^\lambda)$. Finally, it is obvious that $\alpha(t)$ is a $1$-parameter subgroup of $H_p(M, \nabla^\lambda)$ which concludes the proof of the lemma. 
\end{proof}

We have a similar result for the family of connections $\tilde\nabla^f$, where $f \in C^\infty(G)$. Denote by $\tilde\tau_t^f$ the $\tilde\nabla^f$-parallel transport along $c|_{[0,t]}$.

\begin{corollary}\label{contenida}
We have that $\tilde\tau_{-t}^f\tau_t \in H_p(M, \nabla^\lambda)$ for all $t$. In particular, we have that $\Hol(\tilde\nabla^f) \subset G$.
\end{corollary}

\begin{proof}
Just by following the argument in the proof of the previous lemma we get the following differential equation
$$\frac{d}{dt}\tilde\tau_{-t}^f\tau_t = -f(c(t))\tilde\tau_{-t}^f\tau_tD_{c'(0)},$$
which has a unique solution
$$\tilde\tau_{-t}^f\tau_t = e^{-F(t)D_{c'(0)}},$$
where $F(t) = \int_0^tf(c(s))\, ds$.
\end{proof}

Notice that $\tilde\tau_{-t}^f\tau_t$ is no longer a $1$-parameter subgroup of $H(M, \nabla^\lambda)$, unless $f$ is constant.

\begin{remark}
Since $\nabla$ and $\tilde\nabla^f$ have the same geodesics, any $\tilde\nabla^f$-affine transformation is a $\nabla$-affine transformation, since it maps geodesics into geodesics, and $\nabla$ is torsion-free. Since $M$ is compact, this implies that any $\tilde\nabla^f$-affine transformation in the connected component is an isometry of $M$ (see \cite[Lemma 3.6]{reggiani-2010}). That is, 
$$\Aff_0(\tilde\nabla^f) \subset \Iso_0(M) \subset \Iso(M).$$

In particular, for any $\varphi \in \Aff_0(\tilde\nabla^f)$, we have 
$$fD = \varphi_*(fD) = (f \circ \varphi) D,$$
since $\varphi$ preserves the torsion tensor of  $\tilde\nabla^f$ and $\Aff_0(\nabla^\lambda) = \Iso_0(M)$ (see \cite[Theorem 1.1]{olmos-reggiani-2012}). This gives an obstruction to the size of the affine group  of $\tilde\nabla^f$ for $f$ a non-constant function. 
\end{remark}

\begin{corollary}
If $\Aff_0(\tilde\nabla^f)$ is transitive on $M$, then $\tilde\nabla^f$ is a canonical connection on $M$.
\end{corollary}

\begin{proof}
It is straightforward from the above remark. In fact, if $\Aff_0(\tilde\nabla^f)$ is transitive on $M$, then $f$ turns out to be invariant under a transitive subgroup of isometries and therefore it must be constant.
\end{proof}

In the case of a flat metric connection with skew-symmetric torsion we can say even more. In fact, we can prove the following theorem, which is a refinement of the Cartan-Schouten theorem \cite{cartan-schouten-1926, agricola-friedrich-2010}.

\begin{theorem}\label{flat}
Let $M$ be a complete, simply connected and irreducible Riemannian manifold. Let $\tilde\nabla$ be a metric connection on $M$ with the same geodesics as the Levi-Civita connection. If $M \neq S^7$ and $\tilde\nabla$ is flat (that is $\tilde R = 0$), then $M$ is a Lie group with a bi-invariant metric and $\tilde\nabla = \nabla^{\pm1}$ is a canonical connection on $M$.
\end{theorem}

Since we assume that $M \neq S^7$, it follows from the Cartan-Schouten theorem \cite{cartan-schouten-1926,agricola-friedrich-2010} that $M$ is a Lie group. So we can keep the notation of this section. Before giving the proof of Theorem \ref{flat} we need the following useful remarks. 

\begin{remark}
Let $X, Y \in \gm_\lambda$ and $\tilde X, \tilde Y$ be the Killing fields induced by $X, Y$ with initial conditions $\tilde X(p) = X$, $\tilde Y(p) = Y$, where we identify $T_pM$ with $\gm_\lambda$ in the natural way. It is a well-known fact that the Levi-Civita connection $\nabla$ and the canonical connection $\nabla^\lambda$ are given by
$$(\nabla_{\tilde X}\tilde Y)_p = \frac{1}{2}[\tilde X, \tilde Y]_p \simeq -\frac{1}{2}[X, Y]_{\gm_\lambda}$$ 
and 
$$(\nabla^\lambda_{\tilde X}\tilde Y)_p = [\tilde X, \tilde Y]_p \simeq -[X, Y]_{\gm_\lambda}.$$ 
See, for instance, \cite{reggiani-2010}. Taking into account these formulas, it is not hard to show that the relation between the difference tensors $D^\lambda = \nabla - \nabla^\lambda$ and $D^\mu = \nabla - \nabla^\mu$ is given by 
$$\frac{\mu}{\lambda}D^\lambda = D^\mu, \qquad \lambda, \mu \in \bbr,\, \lambda \neq 0.$$
In particular, for all $\lambda \neq 0$, we get the two flat canonical connections with difference tensor $D^{\pm 1} = \pm\frac{1}{\lambda}D^\lambda.$
\end{remark}

\begin{remark}\label{local} Let $\tilde\nabla^f = \nabla - fD$. We give an explicit formula for the curvature tensor $\tilde R^f$ of $\tilde\nabla^f$ in local coordinates $x_i$. By abuse of notation we denote the coordinate vector fields by $i = \partial/\partial x_i$. It is not hard to check that the expression for the $\tilde\nabla^f$-curvature is
$$\tilde R_{i,j}^f = R_{i,j} + f^2[D_i,D_j] + f\left([\nabla_j, D_i] - [\nabla_i, D_j]\right) + \frac{\partial f}{\partial x_j}D_i - \frac{\partial f}{\partial x_i}D_j,$$
where $[\nabla_j, D_i]k = \nabla_j(D_ik) - D_i(\nabla_jk)$. By making use of the fact that there are two flat canonical connections with $f \equiv \pm\frac{1}{\lambda}$, we get 
$$[\nabla_j, D_i] - [\nabla_i, D_j] = 0.$$
So, the above formula is simplified to 
$$\tilde R^f_{i,j} = R_{i, j} + f^2[D_i, D_j] + \frac{\partial f}{\partial x_j}D_i - \frac{\partial f}{\partial x_i}D_j.$$
\end{remark}

\begin{proof}[Proof of Theorem \ref{flat}]
  For each $p \in G$ we consider the Lie subalgebra $\gh_p \subset \so(\gg)$ defined by $\gh_p = \sspan\{\tilde D_v: v \in T_pG\}$ (algebraic span), with the usual identifications (see Remark \ref{rem:lit} and \cite{agricola-friedrich-2004}).

If $\gh_p \neq \so(\gg)$ for all $p \in G$, then $\tilde D$ is a scalar multiple of $D$ at each point (this follows from the skew-torsion holonomy theorem) and therefore $\tilde\nabla = \tilde\nabla^f$ for some $f \in C^\infty(G)$. On the other hand, if there is a $p \in G$ such that $\gh_p = \so(\gg)$ then $G$ has constant sectional curvatures (see \cite{agricola-friedrich-2010}) and it must be a sphere, $\tilde G = \Spin(3) = S^3$ (universal cover). But in the $3$-dimensional case, there is only one algebraic $3$-form, up to a multiple. So, $\tilde\nabla = \tilde\nabla^f$ for some $f \in C^\infty(G)$.

Then we may assume that $\tilde\nabla$ has the form $\tilde\nabla = \tilde\nabla^f$ for some $f \in C^\infty(G)$.

Now, the previous remark  gives a (nonlinear) system of partial differential equations for a flat connection $\tilde\nabla^f$,
\begin{equation*}\tag{*}
0 = R_{i, j} + f^2[D_i, D_j] + \frac{\partial f}{\partial x_j}D_i - \frac{\partial f}{\partial x_i}D_j, \qquad i \neq j.
\end{equation*}

If we show that system (*) does not admit a non-constant solution, then  we will prove that $\tilde\nabla^f$ is a canonical connection. Since there are two constant solutions $f \equiv \pm\frac{1}{\lambda}$, we have that $R_{i,j} = -\frac{1}{\lambda^2}[D_i, D_j]$, and the above equation becomes
$$0 = \left(f^2 - \frac{1}{\lambda^2}\right)[D_i, D_j] + \frac{\partial f}{\partial x_j}D_i - \frac{\partial f}{\partial x_i}D_j.$$
Now, since $D$ induces a simple orthogonal Lie bracket in each tangent space (by the skew-torsion holonomy theorem), we can always choose an index  couple $i,j$ such that $D_i, D_j$ and $[D_i, D_j]$ is a linear independent set (locally). So, $f^2 - \frac{1}{\lambda^2} \equiv 0$ and therefore $f \equiv \pm \frac{1}{\lambda}$.
\end{proof}

Next, we prove that the holonomy group of a generic metric connection with skew-symmetric torsion $\tilde\nabla$ on $G$ coincides with the Riemannian holonomy.

\begin{theorem}\label{hol=G}
Let $G$ be a simple, simply connected Lie group endowed with a bi-invariant metric and let $\gg$ be the Lie algebra of $G$. Let $\tilde\nabla$ be a metric connection with skew-symmetric torsion on $G$. If $\tilde\nabla$ is not flat and $H_e(G, \tilde\nabla) \neq \SO(\gg)$, then $\Hol(\tilde\nabla) = G$.
\end{theorem}

\begin{proof}
Since $H_e(G, \tilde\nabla) \neq \SO(\gg)$ we may assume that $\tilde\nabla = \tilde\nabla^f$ for some $f \in C^\infty(G)$. In fact, this is done in the proof of Theorem \ref{flat}. Since $\tilde\nabla$ is not flat, from Theorem \ref{flat} we have that $\tilde\nabla^f \neq \nabla^{\pm 1}$.

Let $p \in G$ be such that $|f(p)| \neq 1$ and $\grad(f)_p \neq 0$. We choose local coordinates $x_i$ like in Remark \ref{local}. Without loss of generality we can assume  that $\partial /  \partial x_1 = \grad(f)$ is the gradient field of $f$ near $p$ and the coordinate fields $\partial / \partial x_i$ are orthogonal at $p$. Then, the linear map $\tilde R^f_{1, \cdot}: \gg \to \hol(\tilde\nabla^f)$ is injective when restricted to the normal space to $\grad(f)_p$. In fact, by previous computations we have that
$$\tilde R^f_{1,j}|_p = \left(f(p)^2 - \frac{1}{\lambda^2}\right)[D_1,D_j]_p - \|\grad(f)_p\|^2 D_j|_p.$$
Since $D$ induces an orthogonal Lie algebra, $[D_1, D_j]$ is orthogonal to $D_j$ at $p$. Then $\tilde R^f_{1,j} \neq 0$ for all $j \ge 2$, and therefore $\dim\hol(\tilde\nabla^f) \ge \dim \gg - 1$. If $\dim \hol(\tilde\nabla^f) = \dim \gg -1$, then $\hol(\tilde\nabla^f)$ is a codimension $1$ ideal of $\gg$, which is absurd. So, $\hol(\tilde\nabla^f) = \hol(\nabla) = \gg$.

Finally, the connected component of $\Hol(\tilde\nabla^f)$ coincides with $G = \Hol(\nabla)$. From Corollary \ref{contenida} it follows that $\Hol(\tilde\nabla^f)$ is connected and coincides with $G$.
\end{proof}

\section*{Acknowledgements}
The author would like to thank Carlos Olmos for useful discussions and comments on the topics of this article. The author is also grateful to the referee for valuable suggestions which helped with improving the final version of the article. 

This work was supported by CONICET and partially supported by ANCyT, Secyt-UNC and CIEM.

\bibliography{/home/snr/math/bibtex/mybib.bib}{}

\providecommand{\bysame}{\leavevmode\hbox to3em{\hrulefill}\thinspace}
\providecommand{\MR}{\relax\ifhmode\unskip\space\fi MR }
\providecommand{\MRhref}[2]{%
  \href{http://www.ams.org/mathscinet-getitem?mr=#1}{#2}
}
\providecommand{\href}[2]{#2}
\begin{thebibliography}{CDSO11}

\bibitem[AF04]{agricola-friedrich-2004}
I.~Agricola and Th. Friedrich, \emph{On the holonomy of connections with
  skew-symmetric torsion}, Math. Ann. \textbf{328} (2004), no.~4, 711--748.

\bibitem[AF10]{agricola-friedrich-2010}
\bysame, \emph{A note on flat metric connections with antisymmetric torsion},
  Differential Geom. Appl. \textbf{28} (2010), no.~4, 480--487.

\bibitem[Agr06]{agricola-2006}
I.~Agricola, \emph{The {Srn\'{\i }} lectures on non-integrable geometries with
  torsion}, Arch. Math. (Brno) \textbf{42} (2006), no.~5, 5--84.

\bibitem[AS53]{ambrose-singer-1953}
W.~Ambrose and I.~M. Singer, \emph{A theorem on holonomy}, Trans. Amer. Math.
  Soc. \textbf{75} (1953), 428--443.

\bibitem[BCO03]{berndt-console-olmos-2003}
J.~Berndt, S.~Console, and C.~Olmos, \emph{Submanifolds and holonomy}, Research
  Notes in Mathematics, vol. 434, Chapman \& Hall/CRC, Boca Raton, 2003.

\bibitem[Ber55]{berger-1955}
M.~Berger, \emph{Sur les groupes d'holonomie homog\`enes de vari\'et\'es \`a
  connexion affine et des vari\'et\'es riemanniennes}, Bull. Soc. Math. France
  \textbf{83} (1955), 270--330.

\bibitem[CDSO11]{console-discala-olmos-2011}
S.~Console, A.~J. Di~Scala, and C.~Olmos, \emph{A {B}erger type normal holonomy
  theorem for complex submanifolds}, Math. Ann. \textbf{351} (2011), no.~1,
  187--214.

\bibitem[CS26]{cartan-schouten-1926}
\'E. Cartan and J.~A. Schouten, \emph{On {R}iemannian geometries admitting an
  absolute parallelism}, Proceedings Amsterdam \textbf{29} (1926), 933--946.

\bibitem[Nag13]{nagy-2007}
P.-A. Nagy, \emph{Skew-symmetric prolongations of {L}ie algebras and
  applications}, J. Lie Theory \textbf{23} (2013), no.~1, 1--33.

\bibitem[Olm93]{olmos-1993}
C.~Olmos, \emph{Isoparametric submanifolds and their homogeneous structures},
  J. Differential Geom. \textbf{38} (1993), no.~2, 225--234.

\bibitem[Olm05a]{olmos-2005}
\bysame, \emph{A geometric proof of the {B}erger holonomy theorem}, Ann. of
  Math. \textbf{161} (2005), no.~1, 579--588.

\bibitem[Olm05b]{olmos-2005b}
\bysame, \emph{On the geometry of holonomy systems}, Enseign. Math. \textbf{51}
  (2005), 335--349.

\bibitem[OR12]{olmos-reggiani-2012}
C.~Olmos and S.~Reggiani, \emph{The skew-torsion holonomy theorem and naturally
  reductive spaces}, J. Reine Angew. Math. \textbf{664} (2012), 29--53.

\bibitem[Reg10]{reggiani-2010}
S.~Reggiani, \emph{On the affine group of a normal homogeneous manifold}, Ann.
  Global Anal. Geom. \textbf{37} (2010), no.~4, 351--359.

\bibitem[Sim62]{simons-1962}
J.~Simons, \emph{On the transitivity of holonomy systems}, Ann. of Math.
  \textbf{76} (1962), 213--234.

\bibitem[Tho91]{thorbergsson-1991}
G.~Thorbergsson, \emph{Isoparametric foliations and their buildings}, Ann. of
  Math. \textbf{133} (1991), no.~2, 429--446.

\end{thebibliography}
\bibliographystyle{amsalpha}

\end{document}